\documentclass[a4paper,10pt]{amsart}
\usepackage[utf8]{inputenc}
\usepackage{amsmath,amssymb,amsthm,mathtools,leftindex}
\usepackage{xcolor}
\usepackage{enumerate}
\usepackage{tikz}
\usepackage{tikz-cd}
\usepackage[all]{xy}
\usepackage{float}
\usepackage{soul}

\usepackage{listings}

\usepackage{hyperref}
\hypersetup{hidelinks}

\usepackage{stmaryrd}

\theoremstyle{plain}

\newtheorem{thm}{Theorem}[section]
\newtheorem{corollary}[thm]{Corollary}

\newtheorem{proposition}[thm]{Proposition}
\newtheorem{theorem}[thm]{Theorem}

\newtheorem*{conjecture*}{Conjecture}

\theoremstyle{definition}

\newtheorem{definition}[thm]{Definition}

\newtheorem{example}[thm]{Example}
\newtheorem{question}[thm]{Question}

\newcommand{\Hom}{\operatorname{Hom}}

\newcommand{\bbF}{\mathbb{F}}
\newcommand{\bbZ}{\mathbb{Z}}

\renewcommand{\epsilon}{\varepsilon}

\numberwithin{equation}{section}

\title[Primes represented by quadratic forms and the Weil abscissa]{Primes represented by quadratic forms and the Weil abscissa of abelian profinite groups}

\author[Jann]{Martin Jann}
\address{}
\email{}
\author[Kionke]{Steffen Kionke}
\address{FernUniversit\"at in Hagen, Fakult\"at f\"ur Mathematik und Informatik, 58084 Hagen, Germany}
\email{steffen.kionke@fernuni-hagen.de}


\date{\today}

\begin{document}

\begin{abstract}
Here we show that the Weil abscissa of 
the procyclic groups  $\prod_{p \in S} \bbZ_p$ equals $2$ for three sets $S$:  (i) the set of primes $p \equiv 1 \bmod 3$, (ii) the set of primes $p \equiv 1 \bmod 4$ and (iii) the set of primes $p \equiv 1,3 \bmod 8$. Our argument is based on the observation that integers all of whose prime factors lie in $S$ can be represented by a suitable binary quadratic form, which allows us to use a theorem of Iwaniec to exhibit a minorant for the Weil representation zeta function.
\end{abstract}

\maketitle

\section{Introduction}

The \emph{Weil abscissa} $\alpha(G)$ of a profinite group $G$ is an intricate invariant that quantifies how the number of absolutely irreducible representations of $G$ over finite fields $\mathbb{F}_{p^j}$ increases with the size of the fields. It is the abscissa of convergence of the Weil representation zeta function defined in \cite{corob2024weil}. 
For a finitely generated profinite \emph{abelian} group $G$, the \textit{Weil representation zeta function} admits the following simple definition
	\begin{align*}
		\zeta_G^W(s)=\exp\bigg(\sum_{p\in\mathcal{P}}\sum_{j=1}^{\infty} |\text{Hom}(G,\mathbb{F}_{p^j}^\times)|\frac{p^{-sj}}{j}\bigg),
	\end{align*}
 for $s \in \mathbb{C}$; here $\mathcal{P}$ denotes the set of all prime numbers and $|\text{Hom}(G,\mathbb{F}_{p^j}^\times)|$ the number of continuous homomorphisms from $G$ to $\mathbb{F}_{p^j}^\times$. However, even for abelian groups the Weil abscissa is difficult to compute, since its exact value depends in a subtle way on the distribution of primes in arithmetic progressions; see \cite{corob2024weil,kionke2024}. The purpose of this short note is to calculate the Weil abscissa of three procyclic groups building on ideas of the master thesis of the first author (written at FernUniversit\"at in Hagen in 2025). 

\begin{thm}\label{thm:main}
The procyclic groups
\[
    \prod_{p \equiv 1 \bmod 3} \bbZ_p, \quad \prod_{p \equiv 1 \bmod 4} \bbZ_p, \quad \prod_{p \in \{1,3\} \bmod 8} \bbZ_p
\]
have Weil abscissa $2$.
\end{thm}
Let $S$ be a non-empty subset of $\mathcal{P}$ and let $H_S=\prod_{p\in S} \mathbb{Z}_p$; here $\bbZ_p$ denotes the additive group of $p$-adic integers. Then $H_S$ is a procyclic group; this means, it is a profinite group that contains a dense cyclic subgroup.
Conjecture A of \cite{kionke2024} predicts three equalities, one of which can be stated as follows
\begin{conjecture*}
    The shifted partial Riemann zeta function $\zeta_S(s-1):= \prod_{p \in S} \frac{1}{1-p^{-s+1}}$ has the same abscissa of convergence as $\zeta^W_{H_S}(s)$.
\end{conjecture*}
If $S = \{p \in \mathcal{P} \mid p \equiv a \bmod d\}$ is the set of primes in an arithmetic progression with $\gcd(a,d)=1$, then $\zeta_S(s)$ has abscissa $1$ - this can be seen for instance using the Siegel–Walfisz theorem \cite{walfisz1936} - and hence our theorem confirms the conjecture for these three groups.

\medskip

Currently the conjecture is known only for very big or very small sets $S$.
First, consider sets $S \subset\mathcal{P}$ for which $\zeta_{\overline{S}}$ converges for a real number $s<1$, where $\overline{S}=\mathcal{P}\setminus S$. These sets are called \textit{thick}. If $S$ is a thick set of primes, then the second author proved $\alpha(H_S)=2$; see  \cite[Theorem 1.3]{kionke2024}. Second, the conjecture is known to hold if $\zeta_S(s)$ has abscissa $0$ (\cite[Lemma 2.1]{kionke2024}). For certain exponentially growing sequences $S$ of prime numbers this was already observed in \cite[Remark 5.4]{corob2024weil}.
Therefore, the question of the value of the Weil abscissa is of special interest in "in-between" cases.

\medskip

In this paper we analyze cases in which $S$ is the set of prime numbers that are part of one or more arithmetic progressions. To prove our main result, we will exploit the fact that primes in some arithmetic progressions are representable by certain binary quadratic forms and apply a theorem of Iwaniec \cite{Iwaniec1972}. The methods can be extended to prove general results about certain unions of arithmetic progressions. However, the method is limited and we emphasize it does not apply to general arithmetic progressions. In particular, the following question remains open.

\begin{question}
Do the procyclic groups $\prod_{p \equiv 2 \bmod 3} \bbZ_p$ and $\prod_{p \equiv 3 \bmod 4} \bbZ_p$ have Weil abscissa $2$?
\end{question}

\section{Quadratic forms and primes}
Let $S$ be a set of primes. Given a natural number $n$, we write $n_S$ to denote the largest divisor of $n$ all of whose prime factors lie in $S$.
We observe that
$|\text{Hom}(H_S,\mathbb{F}_{p^j}^\times)|=(p^j-1)_S$ for all $p\in \mathcal{P}$. Here we aim to describe sets of primes $p$ such that $(p-1)_S$ is close to $p$ for three specific sets $S_1$, $S_2$ and $S_3$.

\medskip

It is well-known that primes in arithmetic progressions can sometimes be detected using representability by a quadratic form.
\begin{definition}[cf.~\cite{Cox1997}]
    Let $a,b,c \in \mathbb{Z}$ be constants. Then $f:\mathbb{Z}\times\mathbb{Z}\to \mathbb{Z}$ with $f(x,y)= ax^2+bxy+cy^2$ is called a \emph{binary quadratic form} and $\Delta_f=b^2-4ac$ its \emph{discriminant}. 
    A number $n \in \mathbb{N}$ is said to be \emph{representable by $f$} if there are $x,y\in \mathbb{Z}$ with $f(x,y)=n$.  In this case we say that $f$ \emph{represents} $n$.
\end{definition}

\begin{example}
Let $n$ be a positive integer. The quadratic form $f_n(x,y) = x^2+ny^2$ has the discriminant $-4n$.
\end{example}

Representability of integers by quadratic forms is a classical topic in number theory. The sum of two squares problem is a famous example: which primes are sums of two squares? It was already known to Fermat that exactly $2$ and all primes in $S_1=\{p\in \mathcal{P} \mid p \equiv 1 \bmod 4\}$ are the sum of two squares. We will need the following refined result on representations of integers as a sum of two \emph{coprime} squares:
\begin{proposition}[{see \cite[Corollary 3.2.2]{Fine2007}}]\label{prop:1mod4}
    Let $n \in \mathbb{N}$.  Then $n=x^2+y^2$ for coprime integers $x,y$ if and only if the prime decomposition of $n$ is $n=2^\delta \prod_{i=1}^k p_i^{\alpha_i}$ where $\delta \in \{0,1\}$ and $p_1,\dots, p_k \in S_1$.
\end{proposition}
This result immediately implies
\begin{corollary}\label{cor:1}
Define $P_{S_1}=\{p\in\mathcal{P} \mid p=x^2+y^2+1,\,x,y\in \mathbb{Z}, \; \gcd(x,y)=1\}$. All primes $p \in P_{S_1}$ satisfy $(p-1)_{S_1}\geq (p-1)/2$.
\end{corollary}

It is known that Proposition \ref{prop:1mod4} can be extended to the quadratic form $f(x,y)=x^2+ny^2$ for some $n\in \mathbb{N}$ in certain cases \cite{Cox1997}. If $m=x^2+ny^2$  for $x,y \in \mathbb{Z}$ with $\gcd(x,y)=1$, then $\gcd(m,y)=1$. Therefore, $y$ is a unit modulo $m$ and $-n$ is a quadratic residue modulo $m$ and modulo every prime divisor of $m$. 
The only possible prime divisors of $m$ that do not divide $n$ are the primes $p$ for which the corresponding Legendre symbol $\begin{pmatrix}
    \frac{-n}{p}
\end{pmatrix}$ is equal to $1$. This condition forces the possible prime divisors to lie in certain arithmetic progressions. We will present the results for $n=2$ and $n=3$.

\begin{proposition}
    Let $n \in \mathbb{N}$ and $S_2=\{p\in \mathcal{P} \mid p \equiv 1 \bmod 8 \text{ or } p \equiv 3 \bmod 8\}$. If $n=x^2+2y^2$ for coprime $x,y \in \mathbb{Z}$, then $n=2^\delta \prod_{i=1}^k p_i^{\alpha_i}$, where $\delta \in \{0,1\}$ and $p_1,\dots,p_k\in S_2$.
\end{proposition}
\begin{proof}
    First, we observe that for any even integer $n=x^2+2y^2$, where $\gcd(x,y)=1$, $x$ must be an even integer. Therefore, $n$ has a residue of $2$ modulo $4$ and is only divisible by $2$, not $4$. For odd primes the result is given by an old theorem of Fermat \cite[p. 7]{Cox1997}:
    Odd prime divisors $p$ of $n$ must satisfy the condition \begin{align*}
        1=\begin{pmatrix}
    \frac{-2}{p}
\end{pmatrix}=\begin{pmatrix}
    \frac{-1}{p}
\end{pmatrix}\begin{pmatrix}
    \frac{2}{p}
\end{pmatrix}=(-1)^{\frac{p-1}{2}}(-1)^{\frac{p^2-1}{8}}. 
    \end{align*} 
    This implies, as one readily verifies, that all odd prime divisors of $n$ lie in~$S_2$.
\end{proof}
\begin{corollary}\label{cor:2}
Define $P_{S_2}=\{p\in\mathcal{P} \mid p=x^2+2y^2+1,\,x,y\in \mathbb{Z}, \; \gcd(x,y)=1\}$.
All primes $p \in P_{S_2}$ satisfy $(p-1)_{S_2}\geq (p-1)/2$.
\end{corollary}

\begin{proposition}
    Let $n \in \mathbb{N}$ and $S_3=\{p\in \mathcal{P} \mid p \equiv 1 \bmod 3\}$. If $n=x^2+3y^2$ for coprime $x,y \in \mathbb{Z}$, then $n=2^{\delta_1} 3^{\delta_2} \prod_{i=1}^k p_i^{\alpha_i}$, where $\delta_1 \in \{0,1,2\}$,  $\delta_2 \in \{0,1\}$, $p_1,\dots,p_k\in S_3$.
\end{proposition}
\begin{proof}
  First, we observe that for any even integer $n=x^2+3y^2$, where $\gcd(x,y)=1$, both $x$ and $y$ are odd integers.
    Since all odd squares have a residue of $1$ modulo $8$, $n$ cannot be divisible by $8$.
    Regarding the divisor $3^{\delta_2}$, there are two cases.
    First, if $y$ is divisible by $3$, then $x$ is not and $\delta_2=0$. Second, if $x$ is divisible by $3$, then $y$ is not, and thus, $\delta_2=1$. For odd primes the result is again given by an old theorem of Fermat \cite[p. 7]{Cox1997}:
      According to the law of quadratic reciprocity, odd prime divisors $p\neq 3$ must satisfy the following condition: \begin{align*}
        1=\begin{pmatrix}
    \frac{-3}{p}
\end{pmatrix}=\begin{pmatrix}
    \frac{-1}{p}
\end{pmatrix}\begin{pmatrix}
    \frac{3}{p}
\end{pmatrix}=(-1)^{\frac{p-1}{2}} \begin{pmatrix}
    \frac{3}{p}
\end{pmatrix}=\begin{pmatrix}
    \frac{p}{3}
\end{pmatrix}. 
    \end{align*} 
   Since $1$ is the only quadratic residue of a prime $p\neq3$ modulo $3$, the only possible odd prime divisors of $n$ are the primes in $S_3$. 
\end{proof}
As before we deduce
\begin{corollary}\label{cor:3}
Define $P_{S_3}=\{p\in\mathcal{P} \mid p=x^2+3y^2+1,\,x,y\in \mathbb{Z}, \; \gcd(x,y)=1\}$. 
All primes $p \in P_{S_3}$ satisfy $(p-1)_{S_3}\geq (p-1)/12$. 
\end{corollary}
Next, we will need some results from the theory of quadratic forms.
\begin{definition}[see \cite{Cox1997}]
    Two binary quadratic forms, $f(x,y)$ and $g(x,y$), are said to be \emph{(properly) equivalent} if there are $a,b,c,d \in \mathbb{Z}$ such that $ad-bc=1$ and
    \begin{align*}
        f(x,y)=g(ax+by,cx+dy) \text{ for all } x,y \in \mathbb{Z}. 
    \end{align*}
    The \emph{genus of a quadratic form} $f(x,y)$ is the set of all binary quadratic forms that represent the same numbers in $(\mathbb{Z}/|\Delta_f|\mathbb{Z})^*$. It is denoted by $R_f$.
\end{definition}

The equivalence of two binary quadratic forms is an equivalence relation \cite{Cox1997}. All forms in an equivalence class represent the same set of natural numbers. However, this is not true for the genus of a quadratic form. Two forms in the same genus are not necessarily equivalent.

\begin{example}
Consider $f_n(x,y)=x^2+ny^2$ for $n \in \mathbb{N}$.
    \begin{enumerate}
        \item For $n=1,2,3$ the genus of the form $f_n$ contains only one equivalence class. In fact, there is only one equivalence class of primitive positive definite forms of discriminant $-4,-8,-12$ (see \cite[Ch. I, Theorem 2.18]{Cox1997}).
        \item For $n=14$ the forms $x^2+14y^2$ and $2x^2+7y^2$ lie in the same genus. However, $71=2\cdot2^2+7\cdot3^2$ is not representable by $x^2+14y^2$, meaning the two quadratic forms are not equivalent \cite[Ch. I \S2.C]{Cox1997}.
    \end{enumerate}
\end{example}

\section{Proof of Theorem \ref{thm:main}}
For a binary quadratic form $f$, let $P_f$ be the set of primes of the form \[p=g(x,y)+1,\]
where $g\in R_f$ and coprime integers $x,y \in \mathbb{Z}$.
For $N \in \mathbb{R}_+$, let $\pi(N;f)$ denote the number of primes $p \in P_f$ with $p \leq N$.
We note that for the quadratic forms $f_1=x^2+y^2$, $f_2 = x^2+2y^2$, $f_3=x^2+3y^2$, the genus consists of exactly one equivalence class, hence that set $P_{S_n}$ defined above is exactly $P_{f_n}$ for $n \in \{1,2,3\}$.

Our main tool is the following theorem of Iwaniec:
\begin{theorem}[Iwaniec 1972; {\cite[Theorem 1]{Iwaniec1972}}]\label{thm:Iwaniec}
    Let $f(x,y)=ax^2+bxy+cy^2$ be a binary quadratic form with the following properties: $a>0$, $\gcd(a,b,c)=1$ and $\Delta_f$ is not a perfect square. Then there is a positive real constant, $c_1$, such that
    \begin{align*}
        c_1\frac{N}{\ln(N)^{3/2}}(1+o(1))+O\bigg(\frac{N}{\ln(N)^{5}}\bigg) 
        &< \pi(N;f) .
    \end{align*}
\end{theorem}
In fact, {\cite[Theorem 1]{Iwaniec1972} is more general and is concerned with primes of the form $Bg(x,y)+A$; here we only need $A=1$ and $B=1$. The constant $c_1$ is explicitly described in \cite{Iwaniec1972} by $c_1 = \theta\Psi_{A,B,D}\Omega_{A,B,D}$ and one can verify by unwinding the definitions that all terms in this product are positive.

For us the following consequence is important, which is weaker than Iwaniec's result.

\begin{corollary}\label{cor:divergent}
    Under the conditions of Theorem \ref{thm:Iwaniec}, the sum  $\sum_{ p \in P_f} \frac{1}{p^{1-\epsilon}}$ diverges for all $\epsilon>0$.
\end{corollary}
\begin{proof}
Since the assertion is true otherwise, we can restrict the analysis to the case of $\epsilon<1$.
Fix some $\delta \in (0,1)$. We can replace the $(1+o(1))$ term in Iwaniec's theorem by $\delta$ for all large $N$.
We deduce from Theorem \ref{thm:Iwaniec} that there exist $N_0 \in \mathbb{R}$ so that for all $N\geq N_0$ 
    	\begin{align*}
		 \sum_{ p \in P_f,p \leq N} \frac{1}{p^{1-\epsilon}}&\geq \frac{\pi(N;f)}{N^{1-\epsilon}} 
		 \\&\geq \delta c_1\frac{N^{\epsilon}}{\ln(N)^{3/2}} + O\bigg(\frac{N^\epsilon}{\ln(N)^{5}}\bigg); 
	\end{align*}
    in the first inequality, we simply bound each term in the sum from below by $N^{\epsilon-1}$. Taking the limit as $N$ approaches infinity yields the result.
\end{proof}

We can finally prove our main result by specializing $P_f$ to $P_{S_1}, P_{S_2}$ and $P_{S_3}$.

\begin{proof}[Proof of Theorem \ref{thm:main}]
Let $S$ be any set of primes.
The free procyclic group $\widehat{\bbZ} = \prod_{p \in \mathcal{P}} \bbZ_p$ has Weil abscissa $2$; see \cite[Theorem C (i)]{corob2024weil}. Hence the inequality 
\[\alpha(H_S) \leq 2\] follows immediately using that $H_S$ is a quotient of $\widehat{\bbZ}$. 

\medskip

It remains to establish the lower bound. Let $n \in \{1,2,3\}$ and put $S = S_n$. We exhibit divergent minorants for 
\[\ln\bigl(\zeta_{H_S}^W(2-\epsilon)\bigr) = \sum_{p \in \mathcal{P}}\sum_{j=1}^\infty |\Hom(H_S,\bbF_{p^j})| \frac{p^{-j(2-\epsilon)}}{j}\]
for all $\epsilon>0$. Recall that $|\Hom(H_S,\bbF_{p^j})| = (p^j-1)_S$.
 Ignoring the summands for $j>1$, we see that a minorant for $\ln\bigl(\zeta_{H_S}^W(2-\epsilon)\bigr)$ is $\sum_{p \in \mathcal{P}}(p-1)_S \cdot p^{-2+\epsilon}$.
For each of the three cases,  that set $P_S\subset\mathcal{P}$ is such that  $(p-1)_S\geq (p-1)/c$ for a constant $c$ for all $p \in P_S$ (see Corollaries \ref{cor:1}, \ref{cor:2}, \ref{cor:3}). Then we have asymptotically
\begin{align*}
    \sum_{p \in \mathcal{P}}(p-1)_S \cdot p^{-2+\epsilon}\geq \sum_{p \in P_S} \frac{p-1}{cp^{2-\epsilon}} \gg \sum_{p \in P_S} \frac{1}{p^{1-\epsilon}},
\end{align*}
and the right sum diverges by Corollary \ref{cor:divergent}. In summary, we have $\alpha(H_{S_i})\geq 2$ for $i=1,2,3$. Thus, Theorem~\ref{thm:main} follows.
\end{proof}

\bibliography{MA}
\bibliographystyle{plain}
\end{document}